\documentclass[11pt,a4paper,notitlepage]{article}
\usepackage{graphicx}
\usepackage{cite}
\usepackage{epsfig}
\usepackage{amsmath}
\usepackage{amssymb}
\usepackage{amsthm}
\usepackage{booktabs}
\usepackage{stmaryrd}
\usepackage{url}
\usepackage{longtable}
\usepackage[figuresright]{rotating}
\usepackage[utf8]{inputenc}
\usepackage{geometry}
\usepackage[normalem]{ulem}
\usepackage{lipsum}
\usepackage{tocloft}
\usepackage{listings}
\usepackage{fancyhdr}
\usepackage[ruled,vlined,linesnumbered]{algorithm2e}

\usepackage{pgf,tikz}
\usepackage{mathrsfs}
\usetikzlibrary{arrows}
\usetikzlibrary{automata,positioning}

\theoremstyle{plain}
\newtheorem{thm}{Theorem}[section]
\newtheorem{lem}[thm]{Lemma}

\newtheorem{cor}[thm]{Corollary}

\newtheorem{prob}[thm]{Problem}

\theoremstyle{definition}

\newtheorem{rem}[thm]{Remark}

\title{Composition inverses of the variations of the Baum--Sweet sequence}
\author{Łukasz Merta}
\date{}

\begin{document}
\thispagestyle{empty}

\maketitle

\begin{abstract}

Studying and comparing arithmetic properties of a given automatic sequence and the sequence of coefficients of the composition inverse of the associated formal power series (the formal inverse of that sequence) is an interesting problem. This problem was studied before for the Thue--Morse sequence. In this paper, we study arithmetic properties of the formal inverses of two sequences closely related to the well-known Baum--Sweet sequence. We give the recurrence relations for their formal inverses and we determine whether the sequences of indices at which these formal inverses take value $0$ and $1$ are regular. We also show an unexpected connection between one of the obtained sequences and the formal inverse of the Thue--Morse sequence.

\end{abstract}

\vspace{0.5 cm}


\vspace{0.1 cm}
\section{Introduction}

\noindent Let $k \geq 2$ and let $(a_n)_{n \in \mathbb{N}}$ be an infinite sequence. The sequence $(a_n)_{n \in \mathbb{N}}$ is called $k$-automatic if it is a sequence whose $n$-th term is generated from the base-$k$ expansion of $n$ using a finite algorithm. This is equivalent to the fact that the $k$-kernel of the sequence $(a_n)_{n \in \mathbb{N}}$, namely
$$\mathcal{K}_k((a_n)_{n \in \mathbb{N}}) = \{(a_{k^in+j})_{n \in \mathbb{N}} : i \in \mathbb{N}, \, 0 \leq j < k^i\},$$

\noindent is finite \cite[Theorem 6.6.2]{AS}.

\vspace{0.3 cm}
\noindent One of the best known examples of automatic sequence is the Thue--Morse sequence $(t_n)_{n \in \mathbb{N}}$, whose $n$-th term is equal to the sum of digits in the binary expansion of $n$, taken modulo $2$. The sequence $(t_n)_{n \in \mathbb{N}}$ is $2$-automatic and it satisfies the following recurrence relations:
$$t_{2n} = t_n, \qquad t_{2n+1} = 1 - t_n.$$

\vspace{0.2 cm}
\noindent By definition, $k$-automatic sequences take only finitely many values. We can consider a generalization of the class of $k$-automatic sequences, where the sequences can take infinitely many values as well. Such a class, called the class of $k$-regular sequences, was introduced by Allouche and Shallit \cite{REG, REG2}. An infinite sequence $(u_n)_{n \in \mathbb{N}}$, taking values in a $\mathbb{Z}$-module $M$ is called $k$-regular if there exists a finite number of sequences $(s_n^{(1)})_{n \in \mathbb{N}}, \dots, (s_n^{(m)})_{n \in \mathbb{N}},$ taking values in $M$, such that every subsequence in $\mathcal{K}_k((u_n)_{n \in \mathbb{N}})$ can be written in the form
$$u_{k^in+j} = \sum_{r=1}^ma_rs_n^{(r)}$$

\noindent for some $a_1, \dots, a_r \in \mathbb{Z}$.

\vspace{0.3 cm}
\noindent The class of regular sequences is closed under addition and multiplication, i.e.\ if the sequences $(a_n)_{n \in \mathbb{N}}$ and $(b_n)_{n \in \mathbb{N}}$ are regular, then so are $(a_n+b_n)_{n \in \mathbb{N}}$ and $(a_nb_n)_{n \in \mathbb{N}}$. If the sequence $(a_n)_{n \in \mathbb{N}}$ is $k$-regular, then so is $(a_{pn + q})_{n \in \mathbb{N}}$ for all $p \geq 1$ and $q \geq 0$. Conversely, if the sequences $(a_{pn+q})_{n \in \mathbb{N}}$ are $k$-regular for $p \geq 2$ and $q \in \{0, 1, \dots, p-1\}$, then so is the sequence $(a_n)_{n \in \mathbb{N}}$. Finally, if the sequence $(a_n)_{n \in \mathbb{N}}$ is $k$-regular, then the sequence $(a_n \!\! \mod m)_{n \in \mathbb{N}}$ is $k$-automatic for any $m \geq 1$. All these properties and proofs can be found in \cite{REG}.

\vspace{0.3 cm}
\noindent The main aim of this paper is to consider the following problem. Let $p$ be a prime number. Consider a $p$-automatic sequence $(u_n)_{n \in \mathbb{N}}$ with values in the finite field $\mathbb{F}_p$. We consider a formal power series $U = \sum_{n=0}^{\infty}u_nX^n \in \mathbb{F}_p[\![X]\!]$. It is known that if the sequence $(u_n)_{n \in \mathbb{N}}$ satisfies $u_0 = 0$ and $u_1 \neq 0$, then there exists a unique formal power series $V = \sum_{n=0}^{\infty}v_nX^n \in \mathbb{F}_p[\![X]\!]$ such that $U(V(X)) = V(U(X)) = X$ (see \cite[Theorem 6.1(a)]{LANG}). The sequence $(v_n)_{n \in \mathbb{N}}$ of coefficients of $G$ is then called the formal inverse of the sequence $(u_n)_{n \in \mathbb{N}}$. Moreover, by Christol's theorem \cite[Theorem 12.2.5]{AS}, the sequence $(v_n)_{n \in \mathbb{N}}$ is $p$-automatic as well. Our aim is to study properties of this sequence. 

\vspace{0.3 cm}
\noindent This problem was studied by M.\ Gawron and M.\ Ulas in\cite{PTM} in the case when the sequence $(u_n)_{n \in \mathbb{N}}$ is the Thue--Morse sequence $(t_n)_{n\in \mathbb{N}}$. The authors proved some interesting properties of the formal inverse of the Thue--Morse sequence, denoted by $(c_n)_{n \in \mathbb{N}}$. In particular, they found recurrence relations for the sequence $(c_n)_{n \in \mathbb{N}}$ and an automaton that generates it. Furthermore, they proved that the sequence $(c_n)_{n \in \mathbb{N}}$ contains arbitrarily long sequences of consecutive $0$'s and at most $4$ consecutive $1$'s.  Moreover, they considered the formal power series $C = \sum_{n=0}^{\infty}c_nX^n \in \mathbb{C}[\![X]\!]$ (where the coefficients $c_n \in \mathbb{F}_2$ are regarded as integers $0$ and $1$) and they proved that it is transcendental over $\mathbb{C}(X)$.

\vspace{0.3 cm}
\noindent The authors also considered the characteristic sequences of $0$'s and $1$'s in the sequence $(c_n)_{n \in \mathbb{N}}$, i.e.\ the pair of increasing sequences $(a_n)_{n \in \mathbb
{N}}$ and $(d_n)_{n \in \mathbb{N}}$, satisfying the following equalities:
$$\begin{array}{c}
\{m \in \mathbb{N} : c_m = 1\} = \{a_n : n \in \mathbb{N}\}, \vspace{0.05 cm}\\
\{m \in \mathbb{N} : c_m = 0\} = \{d_n : n \in \mathbb{N}\}.\end{array}$$

\vspace{0.2 cm}
\noindent They proved that the sequence $(a_n)_{n \in \mathbb{N}}$ is $2$-regular and the sequence $(l_n)_{n \in \mathbb{N}}$ is not $k$-regular for any $k \geq 2$. They also managed to prove some interesting properties of the sequence $(a_n)_{n \in \mathbb{N}}$. They studied the behavior of the sequence $(a_{n+1} - a_n)_{n \in \mathbb{N}}$ and they proved that the set
$$\mathcal{A} = \left\{\frac{a_n}{n^2} : n \geq 1\right\}$$

\noindent is dense in $\left[ \frac{1}{6}, \frac{1}{2}\right]$.

\vspace{0.3 cm}
\noindent In this paper, we will follow similar lines to study the formal inverse of another $2$-automatic sequence, namely the Baum--Sweet sequence $(b_n)_{n \in \mathbb{N}}$, whose $n$-th term is equal to $0$ if the binary expansion of $n$ contains a block of $0$'s of odd length and $1$ otherwise (with $b_0 = 1)$. Since the sequence itself does not have a formal inverse, we consider two variations of this sequence:
$$b'_n = \left\{\begin{array}{ll}
0 & \text{if } n = 0, \\
b_n & \text{if } n \geq 1, 
\end{array}\right.
\qquad
b''_n = \left\{\begin{array}{ll}
0 & \text{if } n = 0, \\
b_{n-1} & \text{if } n \geq 1.
\end{array}\right.$$

\noindent These two sequences have the formal inverses $(p_n)_{n \in \mathbb{N}}$ and $(q_n)_{n \in \mathbb{N}}$, respectively. 

\vspace{0.3 cm}
\noindent We start with some basic properties of the Baum--Sweet sequence itself, including the fact that the characteristic sequence of $1$'s is not $k$-regular for any $k$ (which is not true for the Thue--Morse sequence). We then consider a formal power series $B = \sum_{n=0}^{\infty}b_nX^n$. We find the algebraic relation for the series $B$, which allows us to find recurrence relations for the sequences $(p_n)_{n \in \mathbb{N}}$ and $(q_n)_{n \in \mathbb{N}}$. We then discuss properties of these sequences. 

\vspace{0.3 cm}
\noindent We prove that the sequence $(p_n)_{n \in \mathbb{N}}$ is ultimately constant and hence the associated formal power series is a rational function. We then find the recurrence relations for the sequence $(q_n)_{n \in \mathbb{N}}$ and the automaton that generates it. We also prove that the associated formal power series (with coefficients regarded as integers) is transcendental over $\mathbb{C}(X)$. 

\vspace{0.3 cm}
\noindent Moreover, we consider an increasing sequence $(u_n)_{n \in \mathbb{N}}$, satisfying
$$\{u_n : n \in \mathbb{N}\} = \{m \in \mathbb{N} : q_n = 1\}.$$

\noindent We prove that this sequence is $2$-regular and we show that it is closely related to the characteristic sequence $(a_n)_{n \in \mathbb{N}}$ of $1$'s in the formal inverse of the Thue--Morse sequence. This implies that both sequences have very similar properties.

\vspace{0.3 cm}
\noindent Finally, we consider a family of subsequences
$$\{(b_n^{(r)})_{n \in \mathbb{N}} : r \in \mathbb{N}, \, r \geq 2\},$$

\noindent such that $b_n^{(r)} = 0$ if and only if the binary expansion of $n$ contains a block of $0$'s of length not divisible by $r$ and $b_n^{(r)} = 1$ otherwise. These sequences are generalizations of the Baum--Sweet sequence (we recover the original sequence for $r = 2$). We then study properties of these sequences and their formal inverses. These properties turn out to be quite similar to properties in the case $r = 2$.

\vspace{0.3 cm}
\noindent In the last section, we give some open problems related to the results shown in previous sections.

\vspace{0.3 cm}
\noindent This paper is an extended version of one of the chapters of the Master's thesis defended at the Jagiellonian University in 2017 \cite{MT}.

\section{Basic results} \label{sec_basic}

\noindent The Baum--Sweet sequence $(b_n)_{n \in \mathbb{N}}$ is a well-known example of a $2$-automatic sequence, with values $0$ and $1$. It is defined by the following rule: we have $b_n = 0$ if the binary expansion of $n$ contains a block of an odd number of $0$'s and $b_n = 1$ otherwise. The first $20$ terms of the sequence $(b_n)_{n \in \mathbb{N}}$ are:
$$1, 1, 0, 1, 1, 0, 0, 1, 0, 1, 0, 0, 1, 0, 0, 1, 1, 0, 0, 1, 0 \dots$$

\noindent It is clear that the sequence $(b_n)_{n \in \mathbb{N}}$ satisfies the following recurrence relations:
$$b_0 = 1, \qquad b_{2n+1} = b_{4n} = b_n, \qquad b_{4n+2} = 0$$

\noindent for $n \in \mathbb{N}$. Hence its $2$-kernel consists of three subsequences: $b_n$, $b_{2n}$ and $b_{4n+2}$. Therefore, the sequence $(b_n)_{n \in \mathbb{N}}$ is generated by an automaton with $3$ states. It is shown on Figure \ref{bn_auto}.

\begin{figure}[ht!]
\begin{center}
\begin{tikzpicture}[->, shorten >= 1pt, node distance=2.5 cm, on grid, auto]
	\node[state, inner sep=1pt, line width = 1.5pt] (c_1) {$1$};
	\node[state, inner sep=1pt] (c_2) [right=of c_1] {$1$};
	\node[state, inner sep=1pt] (c_3) [right=of c_2] {$0$};
	 \path[->]
	 (c_1) edge [bend left] node {0} (c_2)
	 	   edge [loop above] node {1} (c_1)
	 (c_2) edge [bend left] node {0} (c_1)
	       edge node {1} (c_3)
	 (c_3) edge [loop above] node {0, 1} (c_3);	 
\end{tikzpicture}
\vspace{-0.2 cm}
\caption{The automaton generating the sequence $(b_n)_{n \in \mathbb{N}}$. \label{bn_auto}}
\end{center}
\end{figure}
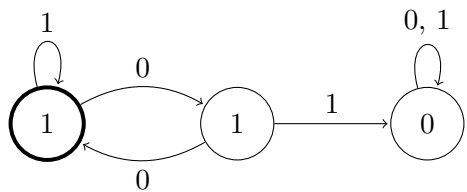

\noindent In the first part of this section, we are going to prove some properties of the sequence $(b_n)_{n \in \mathbb{N}}$. We start with the following theorem:

\begin{thm} \label{baum_01} The Baum--Sweet sequence $(b_n)_{n \in \mathbb{N}}$ contains arbitrarily long sequences of consecutive $0$'s and the maximal number of consecutive $1$'s in this sequence is equal to $2$.
\end{thm}

\begin{proof} Let $k \in \mathbb{N}$ and consider $n = 5 \cdot 2^k$. Then the binary representation of $n$ has the form
$$(n)_2 = 101\underbrace{00 \dots 0}_k,$$

\noindent and the leading digits of the numbers $n, n+1, \dots, n+2^k-1$ are $101$. Hence by the definition of the sequence $(b_n)_{n \in \mathbb{N}}$, we have $b_n = b_{n+1} = \dots = b_{n + 2^k-1} = 0$. This proves the first part of our theorem.

\vspace{0.3 cm}
\noindent The second part of the theorem is a consequence of the fact that if $b_n = b_{n+1} = 1$ and $n > 0$, then $n$ is odd. In order to prove this, suppose that $b_n = b_{n+1} = 1$ and $n > 0$ is even. Therefore, by definition of the sequence $(b_n)_{n \in \mathbb{N}}$, $n$ has the form $n = m \cdot 2^{2k}$, $k \geq 1$ with $2 \nmid m$. We thus have
$$(n+1)_2 = (m)_2\underbrace{00\dots 0}_{2k-1}1,$$

\noindent which implies that $b_{n+1} = 0$, a contradiction. 
\end{proof}

\begin{rem} One can prove that $b_n = b_{n+1} = 1$ if and only if $n$ has the form $n = 4^m - 1,\, m \in \mathbb{N}.$
\end{rem}

\noindent We now consider the sequence of indices at which the sequence $(b_n)_{n \in \mathbb{N}}$ takes value $1$. More precisely, we define the sequence $(l_n)_{n \in \mathbb{N}}$ as an increasing sequence satisfying
$$\{ l_n : n \in \mathbb{N}\} = \{ m \in \mathbb{N} : b_m = 1 \}.$$

\noindent From the definition of the sequence $(b_n)_{n \in \mathbb{N}}$, the sequence $(l_n)_{n \in \mathbb{N}}$ consists of these $m \in \mathbb{N}$ such that the binary expansion of $m$ does not contain any block of consecutive $0$'s of odd length. We are going to prove that this sequence is not $k$-regular for any $k$. Since we could not find the proof of this theorem in the mathematical folklore, we show our own proof.  

\vspace{0.3 cm}
\noindent First, let us recall the definition of the infinite Fibonacci word. Let $\varphi : \{0, 1\}^* \rightarrow \{0, 1\}^*$ be a morphism of monoids such that $\varphi(0) = 01$ and $\varphi(1) = 0$. This morphism is called {\it the Fibonacci morphism}. The Fibonacci word 
$$\varphi^{\omega}(0) = 010010100100101\dots$$

\noindent is the fixed point of $\varphi$ (we extend $\varphi$ to infinite words). It is clear that an infinite word that is a fixed point of $\varphi$ necessarily begins with $0$ and is unique. It turns out that the sequence $(l_n)_{n \in \mathbb{N}}$ and the infinite Fibonacci word are closely related.

\vspace{0.3 cm}
\noindent Write $\varphi_n$ for the $n$-th term of the infinite Fibonacci word. It can be shown that the frequency of $0$'s in the sequence $(\varphi_n)_{n \in \mathbb{N}}$ exists and is equal to $\frac{\sqrt{5} - 1}{2}$ \cite[Example 8.1.3]{AS}. However, we know that for automatic sequences, if the frequency of a given letter exists, it has to be a rational number \cite[Theorem 8.4.5(b)]{AS}. Hence we obtain the following conclusion.

\begin{lem} \label{phi_naut} The sequence $(\varphi_n)_{n \in \mathbb{N}}$ is not $k$-automatic for any $k$.
\end{lem}

\noindent Let us define a sequence of words $(\Lambda_n)_{n \in \mathbb{N}}$ in the following way: 
$$\Lambda_0 = 1, \quad \Lambda_1 = 01, \quad \Lambda_n = \Lambda_{n-2}\Lambda_{n-1}, \quad n \geq 2.$$

\noindent We are going to prove the following result.
\begin{thm} \label{fib_word} The infinite word $\Lambda_0\Lambda_1\Lambda_2\dots$ is equal to the Fibonacci word with the letter $0$ and $1$ interchanged.
\end{thm}

\begin{proof}
\noindent It is sufficient to prove that the infinite word $\Lambda_0\Lambda_1\Lambda_2\dots$ is a fixed point of the morphism $\varphi' : \{0, 1\}^\star \rightarrow \{0, 1\}^\star$ such that
$$\varphi'(0) = 1, \quad \varphi'(1) = 10.$$

\noindent We start with the proof that we have $1\Lambda_{n+1} = \varphi'(\Lambda_n)1$ for all $n \in \mathbb{N}$. We use induction on $n$. For $n = 0$, we have $1\Lambda_1 = 101 = \varphi'(\Lambda_0)1$. By the definition of the sequence $(\Lambda_n)_{n \in \mathbb{N}}$, we have
$$1\Lambda_{n+2} = 1\Lambda_n\Lambda_{n+1} = \varphi'(\Lambda_{n-1})1\Lambda_{n+1} = \varphi'(\Lambda_{n-1})\varphi'(\Lambda_n)1 = \varphi'(\Lambda_{n+1})1,$$

\noindent and the desired equality is proved. As a consequence, using another simple induction on $n$, we get that
\begin{equation} \label{lambda_eq}
\varphi'(\Lambda_0\Lambda_1 \dots \Lambda_n)1 = \Lambda_0\Lambda_1 \dots \Lambda_{n+1}
\end{equation}

\vspace{0.2 cm}
\noindent Since all the words in the sequence $(\Lambda_n)_{n \in \mathbb{N}}$ end with $1$, equation (\ref{lambda_eq}) implies that $\varphi'(\Lambda_0\Lambda_1\dots \Lambda_{n+1})$ contains the word $\Lambda_0\Lambda_1 \dots \Lambda_n$ as a subword for all $n \in \mathbb{N}$. This implies that $\Lambda_0\Lambda_1\Lambda_2 \dots$ is indeed the fixed point of $\varphi'$.
\end{proof}

\noindent We are now ready to prove the following result.

\begin{thm} \label{ln_mod2_eq} We have the following equality:
$$l_n \!\!\!\! \mod 2 = \left\{ \begin{array}{ll}
0 & \text{ if } \, n = 0, \\
1 & \text{ if } \, n = 1, \\
1 - \varphi_{n-2} & \text{ if } \, n \geq 2.
\end{array} \right.$$
\end{thm}

\begin{proof} We define the family of sets $(L_n)_{n \in \mathbb{N}}$ in the following way:
$$L_n = \{ m \in \mathbb{N} : 2^n \leq m < 2^{n+1} \text{  and  } b_m = 1\}.$$

\noindent It is clear that for $k \in \mathbb{N}$, all the numbers in $L_k$ has exactly $k+1$ digits in their binary expansion. Let $k \geq 2$ and $m \in L_k$. The leading digits of $m$ are either $10$ or $11$. If they are equal to $11$, then $m$ has the form $m = 2^k + m'$ where $m' \in L_{k-1}$, and if they are equal to $10$, then $m = 2^k + (m'' - 2^{k-2})$ where $m'' \in L_{k-2}$. Therefore we have the following relation:
\begin{equation} \label{lk_eq}
L_k = (L_{k-2} + (2^k - 2^{k-2})) \cup (L_{k-1} + 2^k), \qquad k \geq 2.
\end{equation}

\vspace{0.2 cm}
\noindent Denote by ${\bf l}$ the infinite word such that the $n$-th letter of ${\bf l}$ is equal to $l_n \!\! \mod 2$. Then by equation (\ref{lk_eq}) and the definition of the sequence $(\Lambda_n)_{n \in \mathbb{N}}$ we have that
$${\bf l} = 01\Lambda_0\Lambda_1\Lambda_2 \dots$$

\noindent and by Theorem \ref{fib_word} we get the desired relation.
\end{proof}

\begin{thm} \label{baum_reg} The sequence $(l_n)_{n \in \mathbb{N}}$ is not $k$-regular for any $k \in \mathbb{N}$. \end{thm}

\begin{proof} Suppose that the sequence $(l_n)_{n \in \mathbb{N}}$ is regular for some $k \in \mathbb{N}$. Then the sequence $(l_n \!\! \mod 2)$ is $k$-automatic. However, this is in contradiction with Lemma \ref{phi_naut} and Theorem \ref{ln_mod2_eq}, since a shift on an automatic sequence is automatic.
\end{proof} 

\noindent In the next result, we consider the formal power series
$$\overline{B} = \sum_{n=0}^{\infty}b_nX^n \in \mathbb{C}[\![X]\!].$$

\noindent We start by finding an algebraic relation for the series $\overline{B}$. From the recurrence relations for the sequence $(b_n)_{n \in \mathbb{N}}$, we have
$$\begin{array}{c}
\displaystyle \overline{B}(X) = \sum_{n=0}^{\infty}b_nX^n = \sum_{n=0}^{\infty}b_{4n}X^{4n} + \sum_{n=0}^{\infty}b_{4n+2}X^{4n+2} + \sum_{n=0}^{\infty}b_{2n+1}X^{2n+1} = \\
\displaystyle = \sum_{n=0}^{\infty}b_nX^{4n} + \sum_{n=0}^{\infty}b_nX^{2n+1} = \overline{B}(X^4) + X\overline{B}(X^2).
\end{array} $$

\noindent Hence the algebraic relation has the form
\begin{equation} \label{b_complex_eq}
\overline{B}(X^4) + X\overline{B}(X^2) - \overline{B}(X) = 0.
\end{equation}

\noindent We have the following result:

\begin{thm} \label{b_transc} 
The series $\overline{B}$ is transcendental over $\mathbb{C}(X)$. 
\end{thm}

\begin{proof} To prove the transcendence of $\overline{B}$, we can use the following theorem by Fatou (see \cite{FATOU}, for another proof see \cite{BC}): a power series whose coefficients take only finitely many values is either rational of transcendental. Hence, it is sufficient to prove that $\overline{B}$ is not a rational function.

\vspace{0.3 cm}
\noindent Assume that there are polynomials $f, g \in \mathbb{C}[X]$ such that $\overline{B}(X) = f(X)/g(X)$ and $g(X) \neq 0$. We may assume that $f$ and $g$ are coprime. From equation (\ref{b_complex_eq}), we obtain
$$f(X^4)g(X^2)g(X) + Xf(X^2)g(X^4)g(X) - f(X)g(X^4)g(X^2) = 0.$$

\noindent We thus have $g(X^4) {\mid} f(X^4)g(X^2)g(X)$. Since $f$ and $g$ are coprime, and therefore so are $f(X^4)$ and $g(X^4)$, we have that $g(X^4) {\mid} g(X^2)g(X)$. This implies that $4\deg g \leq 3\deg g$ and therefore $\deg g = 0$. Hence we have
$$f(X^4) = f(X) - Xf(X^2).$$

\noindent However, this implies that $4 \deg f \leq \max\{\deg f, \, 1 + 2\deg f\}$, which is true only if $\deg f = 0$. This is a contradiction, since $\overline{B}$ is not a constant power series.
\end{proof} 

\noindent Consider the formal power series
$$B = \sum_{n=0}^{\infty}b_nX^n \in \mathbb{F}_2[\![X]\!].$$

\noindent From equation (\ref{b_complex_eq}) we can obtain the algebraic relation for the series $B$. It has the form
\begin{equation} \label{b_eq}
B^4 + XB^2 + B = 0.
\end{equation}

\noindent Since $b_0 = 1$, the composition inverse of $B$ does not exist. Therefore, in order to proceed, we have to slightly change the sequence $(b_n)_{n \in \mathbb{N}}$. There are two natural ways to do it.

\section{Formal inverse of the sequence $(b_n')_{n \in \mathbb{N}}$}

\noindent The first possibility is to change the value of $b_0$ to $0$. We define a new sequence $(b_n')_{n \in \mathbb{N}}$ in the following way:
$$b_n' = \left\{\begin{array}{ll}
0 & \text{ if } n = 0,\\
b_n & \text{ if } n \geq 1.
\end{array} \right.$$

\vspace{0.2 cm}
\noindent Let $C = \sum_{n=0}^{\infty}b_n'X^n$. It is clear that $C = B + 1$. From equality (\ref{b_eq}), we get
\begin{equation} \label{c_eq}
X(C^2 + 1) + (C^4 + C) = 0.
\end{equation}

\noindent Let $P$ be the formal inverse of $C$. Then, after left composing equation (\ref{c_eq}) with $P$ and dividing the obtained equation by $(X + 1)$, we get
\begin{equation} \label{p_eq}
(X + 1)P + X^3 + X^2 + X = 0.
\end{equation}
 
\noindent Let $(p_n)_{n \in \mathbb{N}}$ be the sequence of coefficients of $P$. From equation (\ref{p_eq}), we have
$$\sum_{n=0}^{\infty}p_nX^{n+1} + \sum_{n=0}^{\infty}p_nX^n + X^3 + X^2 + X = 0.$$

\noindent Hence we have $p_0 = 0$, $p_1 = 1$, $p_2 = 0$, $p_3 = 1$ and $p_{n-1} + p_n = 0$ for $n \geq 4$. Therefore the sequence $(p_n)_{n \in \mathbb{N}}$ is ultimately constant and
$$p_n = \left\{\begin{array}{ll}
0 & \text{ if } n = 0 \text{ or } n = 2, \\
1 & \text{ otherwise.}
\end{array} \right.$$

\noindent Note that the obtained sequence has practically nothing in common with the original Baum--Sweet sequence. By Theorem \ref{baum_01} we know that the Baum-Sweet sequence is not ultimately periodic. The sequence $(p_n)_{n \in \mathbb{N}}$ has only two zero terms and an infinite string of consecutive $1$'s, which is in strong contrast with Theorem \ref{baum_01}. Furthermore, the characteristic sequence of $1$'s in the sequence $(p_n)_{n \in \mathbb{N}}$ is obviously $k$-regular for all $k \in \mathbb{N}$. 

\vspace{0.3 cm}
\noindent The sequence $(p_n)_{n \in \mathbb{N}}$ can be considered as a sequence with values in $\mathbb{C}$. We can then define the formal series
$$\overline{P} = \sum_{n=0}^{\infty}p_nX^n \in \mathbb{C}[\![X]\!].$$ 

\noindent Since the sequence $(p_n)_{n \in \mathbb{N}}$ is ultimately constant, we know that $\overline{P}$ is a rational function of the form
$$\overline{P} = \frac{X^3 - X^2 + X}{1 - X}.$$

\vspace{0.2 cm}
\noindent This should again be contrasted with a the similar result for the Baum--Sweet sequence stated in Theorem \ref{b_transc}.
 
\section{Formal inverse of the sequence $(b_n'')_{n \in \mathbb{N}}$} \label{sec_inv2}

\noindent In this subsection, we consider a different way to alter the Baum-Sweet sequence so as to obtain a sequence which has a formal inverse. Instead of changing the first term, we can simply shift all the terms of the sequence by one and add $0$ at the beginning. Formally speaking, we define a new sequence $(b_n'')_{n \in \mathbb{N}}$ which satisfies the following equality:
$$b_n'' = \left\{\begin{array}{ll}
0 & \text{ if } n = 0,\\
b_{n-1} & \text{ if } n \geq 1.
\end{array} \right.$$

\noindent Let $D = \sum_{n=0}^{\infty}b_n''X^n \in \mathbb{F}_2[\![X]\!]$. Then it is clear that $D = XB$. Computing an algebraic equation for $D$ using equation (\ref{b_eq}) in the same way as before, we obtain
\begin{equation} \label{d_eq}
X^3(D^2 + D) + D^4 = 0.
\end{equation}

\noindent Denote by $Q$ the composition inverse of $D$. As before, in order to obtain an algebraic equation for $Q$, we left compose equation (\ref{d_eq}) with $Q$. After dividing the obtained equation by $X$ and multiplying by $Q$, we get the equation
\begin{equation} \label{q_eq}
(X + 1)Q^4 + X^3Q = 0.
\end{equation}

\noindent Denote by $(q_n)_{n \in \mathbb{N}}$ the sequence of coefficients of $Q$. We can use equation (\ref{q_eq}) to find recurrence relations for the sequence $(q_n)_{n \in \mathbb{N}}$. We have the equality
$$\sum_{n=0}^{\infty}q_nX^{4n+1} + \sum_{n=0}^{\infty}q_nX^{4n} + \sum_{n=0}^{\infty}q_nX^{n+3} = 0.$$

\vspace{0.15 cm}
\noindent Hence, after comparing the coefficients, we obtain for $n \in \mathbb{N}$ the relations
\begin{equation} \label{q_recur_eq}
q_{4n} = q_{4n+3} = 0, \qquad q_{4n+1} = q_{4n+2} = q_{n+1}.
\end{equation}

\noindent Using equations (\ref{q_recur_eq}), we can compute all the terms of the sequence $(q_n)_{n \in \mathbb{N}}$ in terms of $q_1$. However, we know that $q_1 = 1$ since the formal series $Q$ is invertible. 

\vspace{0.3 cm}
\noindent From equations (\ref{q_recur_eq}) we can easily deduce that the $2$-kernel of the sequence $(q_n)_{n \in \mathbb{N}}$ has $5$ elements and the sequence is generated by the following automaton:

\vspace{0.3 cm}
\begin{figure}[ht!] 
\begin{center}
\begin{tikzpicture}[->, shorten >= 1pt, node distance=2.5 cm, on grid, auto]
	\node[state, inner sep=1pt, line width = 1.5pt] (c_1) {$0$};
	\node[state, inner sep=1pt] (c_2) [above right=of c_1] {$0$};
	\node[state, inner sep=1pt] (c_3) [below right=of c_1] {$1$};
	\node[state, inner sep=1pt] (c_4) [right=of c_2] {$0$};
	\node[state, inner sep=1pt] (c_5) [right=of c_3] {$1$};
	 \path[->]
	 (c_1) edge node {0} (c_2)
	       edge node [swap] {1} (c_3)
	 (c_2) edge node {0} (c_4)
	       edge [bend right] node {1} (c_5)
	 (c_3) edge node {0} (c_5)
	 	   edge [bend left] node {1} (c_4)
	 (c_4) edge [loop right] node {0, 1} (c_4)
	 (c_5) edge [bend left] node  {0, 1} (c_3);
\end{tikzpicture}
\vspace{-0.2 cm}
\caption{The automaton generating the sequence $(q_n)_{n \in \mathbb{N}}$. \label{qn_auto} } 
\end{center}
\end{figure}
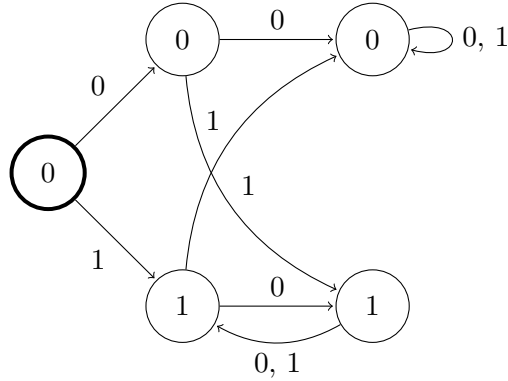

\noindent We now focus on the properties of the sequence $(q_n)_{n \in \mathbb{N}}$ and its characteristic sequence of $1$'s (i.e.\ the sequence of indices at which the sequence $(q_n)_{n \in \mathbb{N}}$ takes value $1$). First, note that the sequence $(q_n)_{n \in \mathbb{N}}$ has arbitrarily long series of consecutive $0$'s, since if $q_{n+1}, q_{n+2}, \dots, q_{n+k}$ is a string of $0$'s of length $k$, then by equations (\ref{q_recur_eq}), $q_{4n+1}, q_{4n+2}, \dots, q_{4n+4k}$ is a string of $0$'s of length $4k$. Furthermore, all the blocks of consecutive $1$'s have length $2$. These properties are similar to the analogous property of the Baum--Sweet sequence shown in Theorem \ref{baum_01}. 

\vspace{0.3 cm}
\noindent As before, we can consider a formal power series
$$\overline{Q} = \sum_{n = 0}^{\infty}q_nX^n \in \mathbb{C}[\![X]\!].$$

\noindent Using the recurrence relations (\ref{q_recur_eq}), we can easily obtain an algebraic relation for the series $\overline{Q}$. We have
\begin{equation} \label{q_complex_eq}
\begin{array}{c}
\displaystyle \overline{Q} = \sum_{i=0}^3 \sum_{n=0}^{\infty}q_{4n+i}X^{4n+i} = \sum_{n=0}^{\infty}q_{n+1}X^{4n+1} + \sum_{n=0}^{\infty}q_{n+1}X^{4n+2} = \vspace{0.1 cm}\\
\displaystyle = \frac{\overline{Q}(X^4)}{X^3} + \frac{\overline{Q}(X^4)}{X^2}
\end{array}
\end{equation}

\noindent where in the last equality we used the fact that $q_0 = 0$. Therefore, the equation above takes the form
$$\overline{Q}(X) = \frac{1 + X}{X^3}\overline{Q}(X^4).$$

\vspace{0.2 cm}
\noindent We have the following theorem.
\begin{thm} \label{q_transc}
The series $\overline{Q}$ is transcendental over $\mathbb{C}(X)$. \end{thm}
 
\begin{proof} We use the same method as in the proof of Theorem \ref{b_transc}. Assume that $\overline{Q}$ is a rational function. Therefore there exist coprime $f, g \in \mathbb{C}[X]$ with $g(X) \neq 0$ and $\overline{Q}(X) = f(X)/g(X)$. From equation (\ref{q_complex_eq}), we have
$$\frac{f(X)}{g(X)} = \frac{(1+X)f(X^4)}{X^3g(X^4)}.$$

\noindent Hence $X^3f(X)g(X^4) = (1+X)f(X^4)g(X)$. Comparing the degrees on both sides of this equation, we get
$$3 + 3\deg g = 1 + 3\deg f,$$ 

\noindent which gives a contradiction.
\end{proof}

\noindent Before we proceed, we recall the definition of the Moser--de Bruijn sequence $(m_n)_{n \in \mathbb{N}}$. It is defined as an increasing sequence of natural numbers $n$ such that $n$ can be expressed as a sum of distinct powers of $4$ (the latter property is equivalent to the fact that $n$ has only digits $0$ and $1$ in its base-$4$ expansion).

\vspace{0.3 cm}
\noindent It is clear that the sequence $(m_n)_{n \in \mathbb{N}}$ satisfies the following recurrence relations:
\begin{equation} \label{mn_recur_eq}
m_0 = 0, \qquad m_{2n} = 4m_n, \qquad m_{2n+1} = 4m_n + 1
\end{equation}

\noindent for all $n \in \mathbb{N}$. Hence the sequence $(m_n)_{n \in \mathbb{N}}$ is $2$-regular. It turns out that the characteristic sequence of $1$'s in the sequence $(q_n)_{n \in \mathbb{N}}$ is related with the Moser--de Bruijn sequence. We are going to show this relation in the following result.

\vspace{0.3 cm}
\noindent Let $(u_n)_{n \in \mathbb{N}}$ be the characteristic sequence of $1$'s in the sequence $(q_n)_{n \in \mathbb{N}}$, that is an increasing sequence satisfying the equality
$$\{u_n : n \in \mathbb{N}\} = \{m \in \mathbb{N} : q_m = 1\}.$$

\noindent The elements of the set $\{u_n : n \in \mathbb{N}\}$ can be characterized in terms of their base-$4$ expansion. It becomes clearly visible, when we change the automaton shown in Figure \ref{qn_auto} so that the input is represented in base $4$. The obtained automaton is shown in Figure \ref{qn_auto_2}. 

\vspace{0.2 cm}
\begin{figure}[ht!] 
\begin{center}
\begin{tikzpicture}[->, shorten >= 1pt, node distance=2.5 cm, on grid, auto]
	\node[state, inner sep=1pt, line width = 1.5pt] (c_1) {$0$};
	\node[state, inner sep=1pt] (c_2) [above right=of c_1] {$0$};
	\node[state, inner sep=1pt] (c_3) [below right=of c_1] {$1$};
	\path[->]
	(c_1) edge node {0, 3} (c_2)
	      edge node [swap] {1, 2} (c_3)
	(c_2) edge [loop right] node {0, 1, 2, 3} (c_2)
	(c_3) edge [loop right] node {0, 1} (c_3)
	      edge [bend right] node [swap] {2, 3} (c_2);
\end{tikzpicture}
\vspace{-0.2 cm}
\caption{Another automaton generating the sequence $(q_n)_{n \in \mathbb{N}}$. \label{qn_auto_2} } 
\end{center}
\end{figure}
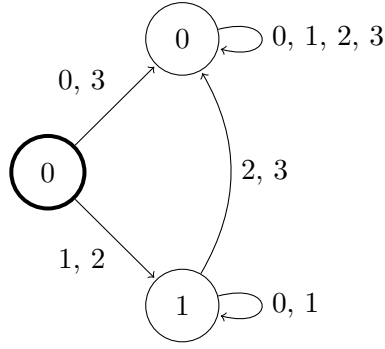

\noindent We have the following result.
\begin{thm} Let $n \in \mathbb{N}$. Then $q_n = 1$ if and only if the base-$4$ expansion of $n$ contains only digits $0$ and $1$, except for the last digit, which can be either $1$ or $2$. 
\end{thm}

\noindent As a consequence, we get the following result.
\begin{cor} \label{un_mn_eq} For all $n \in \mathbb{N}$, we have $u_n = m_n + 1$.
\end{cor}

\noindent We can now easily find a recurrence relations for the sequence $(u_n)_{n \in \mathbb{N}}$ using the above corollary and relations (\ref{mn_recur_eq}).

\begin{thm} \label{un_recur_eq} The sequence $(u_n)_{n \in \mathbb{N}}$ satisfies the following recurrence relations:
$$u_0 = 1, \qquad u_{2n} = 4u_n - 3, \qquad u_{2n+1} = 4u_n - 2.$$
\end{thm} 

\begin{cor} \label{un_2reg} The sequence $(u_n)_{n \in \mathbb{N}}$ is $2$-regular. \end{cor}

\noindent In the following part, we will introduce and prove some properties of the sequences $(m_n)_{n \in \mathbb{N}}$ and $(u_n)_{n \in \mathbb{N}}$. We start with the following theorem, which describes the values of the sequence $(u_{n+1} - u_n)_{n \in \mathbb{N}}$.

\begin{thm} \label{un_differ_eq} We have the equality
$$\{u_{n+1} - u_n : n \in \mathbb{N}\} = \left\{ \frac{1}{3}(1 + 2 \cdot 4^k) : k \in \mathbb{N} \right\}.$$

\noindent Moreover, for each $k \in \mathbb{N}$ we have $u_{n+1} - u_n = \frac{1}{3}(1 + 2\cdot 4^k)$ if and only if $n$ has the form $n = (2m+1)2^k - 1$, $m \in \mathbb{N}$. 
\end{thm}

\begin{proof} Since every $n \in \mathbb{N}$ has a unique representation of the form
$$n = (2m+1)2^k - 1 \qquad m, k \in \mathbb{N},$$ 

\noindent it is sufficient to prove that if $n = (2m + 1)2^k - 1$ for some $m, k \in \mathbb{N}$, then we have $u_{n+1} - u_n = \frac{1}{3}(1 + 2 \cdot 4^k)$. This statement can be proved, using induction on $k$ and the equations in Theorem \ref{un_recur_eq}. \end{proof}

\begin{rem} \label{allseq_eq} From the recurrence relations for the sequences $(t_n)_{n \in \mathbb{N}}$, $(b_n)_{n \in \mathbb{N}}$, $(m_n)_{n \in \mathbb{N}}$ and $(u_n)_{n \in \mathbb{N}}$, we can easily prove that for $n \in \mathbb{N}$, the following equations hold:

\vspace{-0.5 cm}
\begin{align*}
t_{m_{n}} &= t_n, \\
t_{u_{2n}} &= t_{u_{2n+1}} = 1 - t_n, \\
b_{m_{n}} &= \left\{ \begin{array}{ll} 1 & \text{ if } n = 0 \text{ or } n = 2^k, \, k \in \mathbb{N}, \\ 0 & \text{ otherwise,} \end{array} \right. \\
b_{u_{n}} &= \left\{ \begin{array}{ll} 1 & \text{ if } n = 0, \\ 0 & \text{ otherwise.}
\end{array} \right. 
\end{align*}
\end{rem}

\vspace{0.3 cm}
\noindent The properties of the sequence $(u_n)_{n \in \mathbb{N}}$ shown in Corollary \ref{un_2reg}, Theorem \ref{un_differ_eq} and Theorem \ref{allseq_eq} are quite similar to the analogous properties of the characteristic sequence of $1$'s in the sequence $(c_n)_{n \in \mathbb{N}}$, defined as a formal inverse of the Thue--Morse sequence. This suggests that these two sequences are somehow related. In the next part we are going to show such a relation.

\vspace{0.3 cm}
\noindent Denote by $(a_n)_{n \in \mathbb{N}}$ the characteristic sequence of $1$'s in the sequence $(c_n)_{n \in \mathbb{N}}$, with additional $0$ as the $0$-th term. More precisely, the sequence $(a_n)_{n \in \mathbb{N}}$ is an increasing sequence satisfying
$$\{a_n : n \in \mathbb{N}\} = \{m \in \mathbb{N} : c_m = 1\} \cup \{0\}.$$

\vspace{0.2 cm}
\noindent We have $a_0 = 0$, $a_1 = 1$, $a_2 = 2$ and $a_3 = 7$. It was shown in \cite{PTM} that the sequence $(a_n)_{n \in \mathbb{N}}$ is $2$-regular and it satisfies the following recurrence relations:
\begin{align} \label{an_recur_eq}
a_{4n} &= a_{4n-1} + 1, \vspace{0.1 cm} \nonumber \\
a_{4n+1} &= a_{4n-1} + 2, \vspace{0.1 cm} \nonumber \\
a_{4n+2} &= a_{4n-1} + 3, \vspace{0.1 cm} \\
a_{8n+3} &= a_{8n} + 7, \vspace{0.1 cm} \nonumber \\
a_{8n+7} &= 4a_{4n+3} + 3 \nonumber
\end{align}

\noindent for all $n \in \mathbb{N}$. It turns out that this sequence is also related to the Moser--de Bruijn sequence. More precisely, we have the following result.

\begin{thm} \label{an_mn_eq} For all $n \in \mathbb{N}$ we have the following relations:
$$a_{2n} = 2m_n, \qquad a_{2n+1} = 2m_{n+1} - 1.$$
\end{thm} 

\begin{proof} We use induction on $n$. The equations in the statement are true for $n = 0, 1, 2, 3$. From equations (\ref{an_recur_eq}) we get
\begin{align*}
a_{4n} &= a_{4n-1} + 1 = (2m_{2n} - 1) + 1 = 2m_{2n}, \vspace{0.1 cm} \\
a_{4n+1} &= a_{4n} + 1 = 2m_{2n} + 1 = 8m_n + 1 = 2(4m_n + 1) - 1 = 2m_{2n+1} - 1, \vspace{0.1 cm} \\
a_{4n+2} &= a_{4n+1} + 1 = 2m_{2n+1}, \vspace{0.1 cm} \\
a_{8n+3} &= a_{8n} + 7 = 2m_{4n} + 7 = 32m_n + 7 = 8(4m_n + 1) - 1 = 2m_{4n+2} - 1, \vspace{0.1 cm} \\
a_{8n+7} &= 4a_{4n+3} + 3 = 8m_{2n+2} - 4 + 3 = 2m_{4n+4} - 1, 
\end{align*}

\noindent and the desired equations are proved.
\end{proof}

\vspace{0.3 cm}
\noindent By the above theorem and Corollary \ref{un_mn_eq}, we get a relation between the sequences $(a_n)_{n \in \mathbb{N}}$ and $(u_n)_{n \in \mathbb{N}}$.

\begin{cor} \label{an_un_eq} For all $n \in \mathbb{N}$ we have that
$$a_{2n} = 2u_n - 2, \qquad a_{2n+1} = 2u_{n+1} - 3.$$
\end{cor}

\noindent This in an unexpected result. The sequences $(u_n)_{n \in \mathbb{N}}$ and $(a_n)_{n \in \mathbb{N}}$ are closely related despite the fact that they are characteristic sequences of $1$'s in formal inverses of two completely different sequences. 

\vspace{0.3 cm}
\noindent At first glance the Thue--Morse sequence and the Baum--Sweet sequence have nothing in common. In the Thue--Morse sequence, the frequency of $0$'s and $1$'s exists and is equal to $\frac{1}{2}$, we cannot have more than two consecutive $0$'s and two consecutive $1$'s and the characteristic sequences of $0$'s and $1$'s are both $2$-regular. By Theorem \ref{baum_01} and Theorem \ref{baum_reg}, we know that similar properties do not hold for the Baum--Sweet sequence.

\begin{rem} From the equations in Theorem \ref{an_mn_eq}, we can find another set of recurrence relations for the sequence $(a_n)_{n \in \mathbb{N}}$, based on equations (\ref{mn_recur_eq}). The obtained relations have the form

\vspace{-0.9 cm}
\begin{align*}
a_{4n} &= 4a_{2n}, \vspace{0.1 cm} \\
a_{4n+1} &= 4a_{2n} + 1, \vspace{0.1 cm} \\
a_{4n+2} &= 4a_{2n} + 2, \vspace{0.1 cm} \\
a_{4n+3} &= 4a_{2n+1} + 3. 
\end{align*}

\end{rem}
 
\vspace{0.3 cm}
\noindent Since we have a relation between the sequences $(a_n)_{n \in \mathbb{N}}$ and $(u_n)_{n \in \mathbb{N}}$, we can expect that they share even more properties. For example, it was shown in \cite{PTM} that we have
$$\liminf_{n \to \infty}\frac{a_n}{n^2} = \frac{1}{6}, \qquad \limsup_{n \to \infty}\frac{a_n}{n^2} = \frac{1}{2},$$

\noindent and that the set $\{a_n/n^2 : n \in \mathbb{N}\}$ is dense in $[\frac{1}{6}, \frac{1}{2}]$. From Corollary \ref{an_un_eq}, we can immediately get analogous property for the sequence $(u_n)_{n \in \mathbb{N}}$. \

\begin{thm} The following equalities hold:
$$\liminf_{n \to \infty}\frac{u_n}{n^2} = \frac{1}{3}, \qquad \limsup_{n \to \infty}\frac{u_n}{n^2} = 1.$$

\noindent Moreover, the set $\{u_n/n^2 : n \in \mathbb{N}\}$ is dense in $[\frac{1}{3}, 1]$.
\end{thm}

\noindent It is natural to expect that we have a similar relation between the characteristic sequences of $0$'s in the sequences $(q_n)_{n \in \mathbb{N}}$ and $(c_n)_{n \in \mathbb{N}}$ as well. We denote these sequences by $(v_n)_{n \in \mathbb{N}}$ and $(d_n)_{n \in \mathbb{N}}$, respectively. Since we added $0$ to the sequence $(a_n)_{n \in \mathbb{N}}$, we exclude it from the sequence $(d_n)_{n \in \mathbb{N}}$. The sequences $(v_n)_{n \in \mathbb{N}}$ and $(d_n)_{n \in \mathbb{N}}$ are both increasing and they satisfy the following equalities:
$$\begin{array}{c}
\{ v_n : n \in \mathbb{N}\} = \{m \in \mathbb{N} : q_m = 0\}, \vspace{0.1 cm} \\
\{ d_n : n \in \mathbb{N}\} = \{m \geq 1 : c_m = 0\}.
\end{array}.$$

\noindent By definition, we have $\mathbb{N} = \{a_n : n \in \mathbb{N}\} \cup \{d_n : n \in \mathbb{N}\}$ and these sets are disjoint. Moreover, we have
$$\begin{array}{c}
\mathbb{N} = \{2u_n - 2 : n \in \mathbb{N}\} \cup \{2u_n - 3 : n \geq 1\} \, \cup \vspace{0.1 cm} \\
\cup \, \{2v_n - 2 : n \geq 1\} \cup \{2v_n - 3 : n \geq 1\}, \end{array}$$

\noindent and these sets are pairwise disjoint. By the relations from Corollary \ref{an_un_eq}, we know that the elements from the first two sets are the terms of $(a_n)_{n \in \mathbb{N}}$. Therefore the remaining sets contain the terms of $(d_n)_{n \in \mathbb{N}}$. We thus have the following result:

\begin{lem} \label{dn_vn_eq} For all $n \in \mathbb{N}$ we have
$$d_{2n} = 2v_{n+1} - 3, \qquad d_{2n+1} = 2v_{n+1} - 2.$$
\end{lem}

\noindent It was shown in \cite{PTM} that the sequence $(d_n)_{n \in \mathbb{N}}$ is not $k$-regular for any $k$. Hence, from the above lemma, we can get analogous result for the sequence $(v_n)_{n \in \mathbb{N}}$.

\begin{thm} The sequence $(v_n)_{n \in \mathbb{N}}$ is not $k$-regular for any $k$.
\end{thm}

\begin{proof} Follows immediately from Lemma \ref{dn_vn_eq}. \end{proof}

\section{Generalizations}
\noindent For a given natural number $r \geqslant 2$, we consider a sequence $(b_n^{(r)})_{n \in \mathbb{N}}$ defined by the following rule: we have $b_n^{(r)} = 1$ if all the blocks of $0$'s in the binary expansion of $n$ have length divisible by $r$ and $b_n^{(r)} = 0$ otherwise, with $b_0^{(r)} = 1$. For $r = 2$, we recover the original Baum--Sweet sequence.

\vspace{0.3 cm}
\noindent By definition, it is clear that the sequence $(b_n^{(r)})_{n \in \mathbb{N}}$ is $2$-automatic and satisfies the following recurrence relations:
\begin{equation} \label{bnr_recur_eq}\begin{array}{c}
b_0^{(r)} = 1, \qquad b_{2^rn}^{(r)} = b_{2n+1}^{(r)} = b_n^{(r)}, \vspace{0.1 cm} \\
b_{2^in + 2^{i-1}}^{(r)} = 0, \quad i = 2, 3, \dots, r. \end{array}
\end{equation}

\noindent Moreover, it can be shown that the $2$-kernel of the sequence $(b_n^{(r)})$ consists of $r+1$ subsequences, namely
$$(b_n^{(r)})_{n \in \mathbb{N}},\, (b_{2n}^{(r)})_{n \in \mathbb{N}}, \dots, (b_{2^{r-1}n}^{(r)})_{n \in \mathbb{N}},\, (b_{4n+2}^{(r)})_{n \in \mathbb{N}}.$$

\noindent We can expect that properties of the sequence $(b_n^{(r)})_{n \in \mathbb{N}}$ for a given $r \geq 2$ are similar to properties of the original Baum--Sweet sequence. In this section, we will show that this is indeed the case. We start with the following result.

\begin{thm}
The sequence $(b_n^{(r)})_{n \in \mathbb{N}}$ contains arbitrarily long sequences of consecutive $0$'s and the maximal number of consecutive $1$'s is equal to $2$.
\end{thm}

\noindent The proof of this result is analogous to the proof of Theorem \ref{baum_01}.

\vspace{0.3 cm}
\noindent Let $(l_n^{(r)})_{n \in \mathbb{N}}$ be the characteristic sequence of $1$'s in the sequence $(b_n^{(r)})_{n \in \mathbb{N}}$, that is an increasing sequence satisfying the equality
$$\left\{l_n^{(r)} : n \in \mathbb{N} \right\} = \left\{m \in \mathbb{N} : b_m^{(r)} = 1\right\}.$$

\begin{thm}
The sequence $(l_n^{(r)})_{n \in \mathbb{N}}$ is not $k$-regular for any $k \geq 2$.
\end{thm}

\begin{proof} Let $\Sigma = \{x_0, x_1, \dots, x_{r-1}\}$ be an alphabet with $r$ letters. Define $(\Delta_n)_{n \in \mathbb{N}}$ as the sequence of words such that $\Delta_i = x_{i+1}$ for $0 \leq i < r-1$, $\Delta_{r-1} = x_0x_{r-1}$ and $\Delta_n = \Delta_{n-r}\Delta_{n-1}$ for $n \geq r$. Let $\varphi : \Sigma^\star \rightarrow \Sigma^\star$ be a morphism of monoids such that $\varphi(x_i) = \Delta_i$. Then it is easy to verify that we have $\varphi(\Delta_n) = \Delta_{n+1}$ for all $n \in \mathbb{N}$. Moreover, we have the equality
$$\Delta_0\Delta_1\dots\Delta_{n+1} = x_1\varphi(\Delta_0\Delta_1\dots\Delta_n).$$

\noindent As a consequence, we get that
\begin{equation} \label{delta_eq}
\Delta_0\Delta_1\Delta_2\hdots = x_1\varphi(\Delta_0\Delta_1\Delta_2\dots).
\end{equation}

\noindent The morphism $\varphi$ does not have a fixed point. Consider the $r$-th iterate $\mu = \varphi^r$ of $\varphi$. We have that

\vspace{-1 cm}
\begin{align*}
\mu(x_0) &= x_0x_{r-1}, \\
\mu(x_1) &= x_1x_0x_{r-1}, \\
& \;\; \vdots \\
\mu(x_{r-1}) &= x_{r-1}\dots x_1x_0x_{r-1}.
\end{align*}

\noindent Hence the morphism $\mu$ has exactly $r$ fixed points. We are interested in frequencies of letters in the fixed points of $\mu$. Applying $\varphi$ to equation (\ref{delta_eq}), we get that
$$\Delta_0\Delta_1\Delta_2\hdots = x_1x_2\dots x_{r-1}x_0x_{r-1}\mu(\Delta_0\Delta_1\Delta_2\dots),$$

\noindent which implies that the infinite word $\Delta_0\Delta_1\Delta_2\hdots$ can be written in the form
$$\Delta_0\Delta_1\Delta_2 = w\mu(w)\mu^2(w)\mu^3(w)\dots,$$

\noindent where $w = x_1x_2\dots x_{r-1}x_0x_{r-1}$. The morphism $\mu$ is primitive (there exists $k \geq 1$ such that for all $i, j \in \{0, 1, \dots, r-1\}$, $x_i$ occurs in $\mu^k(x_j)$) and therefore in every fixed point of $\mu$ the frequencies of all letters exist and are nonzero and they do not depend on the choice of a fixed point (\hspace{-0.2pt}{\cite[Theorem 8.4.7]{AS}}). Hence, in the infinite word $\Delta_0\Delta_1\Delta_2\hdots$, the frequencies of all letters exist and are nonzero.

\vspace{0.3 cm}
\noindent Denote by ${\bf l}_r$ the infinite word such that the $n$-th letter of ${\bf l}_r$ is equal to $l_n^{(r)} \!\!\! \mod 2$. Using the same method as in the proof of Theorem \ref{ln_mod2_eq}, we get that
$${\bf l}_r = 01\psi(\Delta_0\Delta_1\Delta_2\dots),$$

\noindent where $\psi : \Sigma^\star \rightarrow \{0,1\}^\star$ is a coding such that $\psi(x_0) = 0$ and $\psi(x_i) = 1$ for $0 < i < r$. In order to prove that the sequence $(l_n^{(r)})_{n \in \mathbb{N}}$ is not $k$-regular for any $k \geq 2$, it is sufficient to prove that the frequency of letters in the above infinite word is irrational (\hspace{-0.2pt}{\cite[Theorem 8.4.5]{AS}}). Using an easy induction, we get that
$$\psi(\Delta_n\Delta_{n-1}\dots\Delta_0)01 = \psi(\Delta_{n+r})$$ 

\noindent for all $n \in \mathbb{N}$ and therefore the word $01\psi(\Delta_0\Delta_1\dots\Delta_n)$ can be created from $\psi(\Delta_{n+r})$ after rearranging the letters. Since the frequency of $1$'s in the word $01\psi(\Delta0\Delta1\dots)$ exists, so does the limit
$$g = \lim_{n\to\infty}\frac{|\psi(\Delta_n)|_1}{|\psi(\Delta_n)|}.$$

\vspace{0.2 cm}
\noindent It is easy to check that we have $|\psi(\Delta_n)| = |\psi(\Delta_{n+1})|_1$ for all $n \in \mathbb{N}$. Hence, from the recurrence relations for the sequence $(\Delta_n)_{n \in \mathbb{N}}$, we get that $g$ is a root of the polynomial $x^r + x - 1$. This polynomial has only one root in $(0,1)$ and it is an irrational number, which completes the proof.
\end{proof}

\noindent In the following part, we consider a formal power series
$$\overline{B_r} = \sum_{n=0}^{\infty}b_n^{(r)}X^n \in \mathbb{C}[\![X]\!],$$

\noindent From equations (\ref{bnr_recur_eq}), we obtain the equality
\begin{equation} \label{br_complex_eq}
\overline{B_r}(X^{2^r}) + X\overline{B_r}(X^2) - \overline{B_r}(X) = 0.
\end{equation}

\noindent
We can now use the equation (\ref{br_complex_eq}) and the method used in the proof of Theorem \ref{b_transc} to prove the following result.

\begin{thm} The series $\overline{B_r}$ is transcendental over $\mathbb{C}(X)$.
\end{thm}

\noindent In the following part, we will focus on the formal inverse of the sequence $(b_n^{(r)})_{n \in \mathbb{N}}$. We consider a formal power series
$$B_r = \sum_{n=0}^{\infty}b_n^{(r)}X^n \in \mathbb{F}_2[\![X]\!].$$

\noindent Using the recurrence relations for the sequence $(b_n^{(r)})_{n \in \mathbb{N}}$, we easily obtain the following equality:
\begin{equation} \label{br_eq}
B_r^{2^r} + XB_r^2 + B_r = 0.
\end{equation}

\vspace{0.2 cm}
\noindent We have $b_0^{(r)} = 1$, hence the series $B_r$ does not have a formal inverse. As before, we consider the series $C_r = B_r + 1$ and $D_r = XB_r$. These series are both invertible. Denote by $P_r$ and $Q_r$ the formal inverses of the series $C_r$ and $D_r$, respectively. Let $(p_n^{(r)})_{n \in \mathbb{N}}$ be the sequence of coefficients of $P_r$ and let $(q_n^{(r)})_{n \in \mathbb{N}}$ be the sequence of coefficients of $Q_r$. 

\vspace{0.3 cm} 
\noindent We start with the series $C_r$. From equation (\ref{br_eq}), we get
$$C_r^{2^r} + XC_r^2 + C_r + X = 0.$$

\noindent Hence, after left composing the above equation with $P_r$, we obtain the algebraic relation for the series $P_r$:
\begin{equation} \label{pr_eq}
(X^2 + 1)P_r + X^{2^r} + X = 0.
\end{equation}

\vspace{0.15 cm}
\noindent From equation (\ref{pr_eq}), we can determine all terms of the sequence $(p_n^{(r)})_{n \in \mathbb{N}}$. More precisely, we have the following theorem.

\begin{thm} We have the equality
$$p_n^{(r)} = \left\{\begin{array}{ll}
0 & \text{ if } 2 {\mid} n \text{ and } n < 2^r, \\
1 & \text{ otherwise.}
\end{array} \right.$$
\end{thm}

\begin{rem} Consider the power series $\overline{P_r} = \sum_{n=0}^{\infty}p_n^{(r)}X^n \in \mathbb{C}[\![X]\!].$ Since the sequence $(p_n^{(r)})_{n \in \mathbb{N}}$ is ultimately constant, the series $\overline{P_r}$ is a rational function. Using the above theorem, we can compute that
$$\overline{P_r} = \frac{1}{1-X}\sum_{k=1}^{2^r-1}(-1)^{k-1}X^k.$$
\end{rem}

\noindent We now consider the series $D_r$. Again, from equation (\ref{br_eq}) we have that
$$XD_r^{2^r} + X^{2^r}D_r^2 + x^{2^r}D_r = 0.$$

\noindent Left composing the above equation with $Q_r$, we obtain
$$(X+1)Q_r^{2^r} + X^{2^r-1}Q_r = 0.$$

\noindent After comparing the coefficients at $X^{2^r(n+1)+i}$, $i = -1, 0, 1, \dots, 2^r-2$, we get that the sequence $(q_n^{(r)})_{n \in \mathbb{N}}$ satisfies the following recurrence relations:
$$q_{2^rn+1}^{(r)} = q_{2^rn+2}^{(r)} = q_n^{(r)}, \qquad q_{2^rn + i}^{(r)} = 0, \;\; i = 0, 3, 4, \dots, 2^r-1.$$

\noindent Hence the sequence $(q_n^{(r)})_{n \in \mathbb{N}}$ is generated by the automaton shown in Figure \ref{qnr_auto}. The input is represented in base $2^r$.
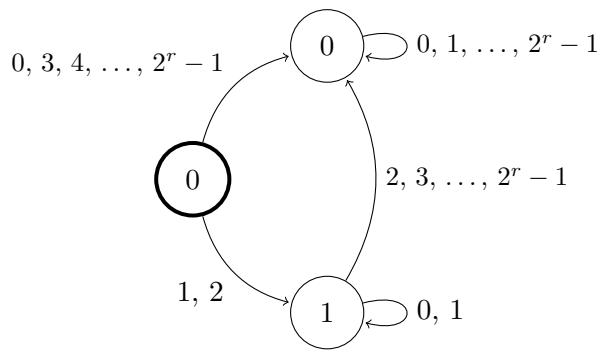
\begin{figure}[ht!] 
\begin{center}
\begin{tikzpicture}[->, shorten >= 1pt, node distance=2.5 cm, on grid, auto]
	\node[state, inner sep=1pt, line width = 1.5pt] (c_1) {$0$};
	\node[state, inner sep=1pt] (c_2) [above right=of c_1] {$0$};
	\node[state, inner sep=1pt] (c_3) [below right=of c_1] {$1$};
	\path[->]
	(c_1) edge [bend left] node {\small 0, 3, 4, $\dots$, $2^r-1$} (c_2)
	      edge [bend right] node [swap] {1, 2} (c_3)
	(c_2) edge [loop right] node {\small 0, 1, $\dots$, $2^r-1$} (c_2)
	(c_3) edge [loop right] node {0, 1} (c_3)
	      edge [bend right] node [swap] {\small 2, 3, $\dots$, $2^r-1$} (c_2);
\end{tikzpicture}
\vspace{-0.2 cm}
\caption{The automaton generating the sequence $(q_n^{(r)})_{n \in \mathbb{N}}$. \label{qnr_auto} } 
\end{center}
\end{figure}

\begin{rem} Consider a sequence $\overline{Q_r} = \sum_{n=0}^{\infty}q_n^{(r)}X^n \in \mathbb{C}[\![X]\!]$. From recurrence relations for the sequence $(q_n^{(r)})_{n \in \mathbb{N}}$ we get that 
$$\overline{Q_r}(X) = \frac{1+X}{X^{2^r-2}}\overline{Q_r}(X^{2^r}).$$

\noindent Moreover, by using the same method as in the proof of Theorem \ref{q_transc}, we can prove that $\overline{Q_r}$ is transcendental over $\mathbb{C}(X)$.
\end{rem}

\vspace{0.2 cm}
\noindent Let $(u_n^{(r)})_{n \in \mathbb{N}}$ be the characteristic sequence of $1$'s in the sequence $(q_n^{(r)})_{n \in \mathbb{N}}$. In the following part, we are going to introduce some properties of this sequence. Let us recall that in the case $r = 2$, the characteristic sequence of $1$'s is closely related to the Moser--de Bruijn sequence. It turns out that we can find such a relation in the general case as well. 

\vspace{0.3 cm}
\noindent We define a sequence $(m_n^{(r)})_{n \in \mathbb{N}}$ as an increasing sequence of natural numbers $n$ such that $n$ can be written as a sum of distinct powers of $2^r$ (in other words, $n$ has only digits $0$ and $1$ in its base-$2^r$ expansion). It is clear that in the case $r = 2$, this sequence is the Moser--de Bruijn sequence.

\vspace{0.3 cm}
\noindent The sequence $(m_n^{(r)})_{n \in \mathbb{N}}$ satisfies the following recurrence relations:
\begin{equation} \label{mnr_recur_eq}
m_0^{(r)} = 0, \qquad m_{2n}^{(r)} = 2^rm_n^{(r)}, \qquad m_{2n+1}^{(r)} = 2^rm_n^{(r)} + 1.
\end{equation}

\noindent From the automaton shown in Figure \ref{qnr_auto}, we get that $q_n^{(r)} = 1$ if and only if $n$ has only digits $0$ and $1$ in its base-$2^r$ expansion, except for the last digit, which can be either $1$ or $2$. As a consequence, we obtain the following result.

\begin{thm} \label{unr_mnr_eq} We have $u_n^{(r)} = m_n^{(r)} + 1$ for $n \in \mathbb{N}$.
\end{thm}
\begin{cor} \label{unr_recur_eq} The sequence $(u_n^{(r)})_{n \in \mathbb{N}}$ satisfies the following recurrence relations:
$$u_0^{(r)} = 1, \qquad u_{2n}^{(r)} = u_n^{(r)} - 2^r + 1, \qquad u_{2n+1}^{(r)} = u_n^{(r)} - 2^r + 2.$$
\end{cor}

\begin{cor} The sequence $(u_n^{(r)})_{n \in \mathbb{N}}$ is $2$-regular.
\end{cor}

\noindent In the following part, we will focus on the properties of the sequence $(u_n^{(r)})_{n \in \mathbb{N}}$. As before, we start with the following theorem, which characterizes the values of the sequence $(u_{n+1}^{(r)} - u_n^{(r)})_{n \in \mathbb{N}}$. This result can be proved using the same method as in the proof of Theorem \ref{un_differ_eq}.

\begin{thm} \label{unr_differ_eq} We have the equality
$$\left\{u_{n+1}^{(r)} - u_n^{(r)} : n \in \mathbb{N}\right\} = \left\{\frac{1}{2^r-1}\left(1 + (2^r-2)\cdot 2^{rk}\right) : k \in \mathbb{N}\right\}.$$

\noindent Moreover, for each $k \in \mathbb{N}$ we have $u_{n+1}^{(r)} - u_n^{(r)} = \frac{1}{2^r-1}(1 + (2^r-2)\cdot2^{rk})$ if and only if $n$ has the form $n = (2m+1)2^k - 1$, $m \in \mathbb{N}$.
\end{thm}

\noindent In section \ref{sec_inv2}, we proved the equalities
$$\liminf_{n\to\infty}\frac{u_n}{n^2} = \frac{1}{3}, \qquad \limsup_{n\to\infty}\frac{u_n}{n^2} = 1.$$

\noindent The proof relies on the connection between the sequence $(u_n)_{n \in \mathbb{N}}$ and the characteristic sequence of $1$'s in the formal inverse of the Thue--Morse sequence. In the general case, we do not have such a connection. However, it is still possible to find the analogous equalities for the sequence $(u_n^{(r)})_{n \in \mathbb{N}}$.

\vspace{0.3 cm}
\noindent We start with the following lemma. It can be proved using induction and equalities in Corollary \ref{unr_recur_eq}.

\begin{lem} For all $n \in \mathbb{N}$ we have that
$$\frac{(n+1)^r + 2^r - 2}{2^r - 1} \leq u_n^{(r)} \leq n^r + 1.$$

\noindent Moreover, for all $m \in \mathbb{N}$ we have
$$u_{2^m-1}^{(r)} = \frac{2^{mr} + 2^r - 2}{2^r - 1}, \qquad u_{2^m}^{(r)} = 2^{mr} + 1.$$
\end{lem}

\vspace{0.1 cm}
\noindent As a consequence, we get the following result.
\begin{thm} We have the equalities
$$\liminf_{n\to\infty}\frac{u_n^{(r)}}{n^r} = \frac{1}{2^r-1}, \qquad \limsup_{n\to\infty}\frac{u_n^{(r)}}{n^r} = 1.$$
\end{thm}

\noindent Let us recall that in the case $r = 2$ we have that the set $\{u_n/n^2 : n \in \mathbb{N}\}$ is dense in the interval $\left[\frac{1}{3}, 1\right]$. In the general case, the analogous statement is also true.

\begin{thm} The set $\{u_n^{(r)}/n^r : n \in \mathbb{N}\}$ is dense in the interval $\left[\frac{1}{2^r-1}, 1\right]$.
\end{thm} 

\noindent To prove this statement, we can use the same method as in the proof of the analogous result for the formal inverse of the Thue--Morse sequence (see \hspace{-0.2pt}{\cite[Theorem 3.8]{PTM}} for details).

\vspace{0.2 cm}
\noindent Denote by $(v_n^{(r)})_{n \in \mathbb{N}}$ the characteristic sequence of $0$'s in the sequence $(q_n^{(r)})_{n \in \mathbb{N}}$. In the following, we are going to prove that the sequence $(v_n^{(r)})_{n \in \mathbb{N}}$ is not $k$-regular for any $k \geq 2$. 

\vspace{0.3 cm}
\noindent First, we consider the sequence $(w_n^{(r)})_{n \in \mathbb{N}}$ where $w_n^{(r)} = v_{n+1}^{(r)} - 1$. Note that by Theorem \ref{unr_mnr_eq} the sequence $(w_n^{(r)})_{n \in \mathbb{N}}$ consists of those natural numbers which do not appear in the sequence $(m_n^{(r)})_{n \in \mathbb{N}}$. Furthermore, we consider a sequence $(s_n^{(r)})_{n \in \mathbb{N}}$ given by
$$s_n^{(r)} = \frac{w_n^{(r)} - n}{2} \!\!\! \mod 2.$$

\noindent We have the following result.

\begin{lem} \label{lemma_sn_1}
We have $s_n^{(r)} = 1$ if and only if $n$ can be written in the form $n = \sum_{i=0}^s(2^{rn_i} - 2^{n_i}) + j$, where $1 < n_0 < \hdots < n_s$ and $j \in \{0, 1, \dots, 2^r - 3\}.$
\end{lem}

\begin{proof} Let us note that for all $n \in \mathbb{N}$ we have
\begin{equation} \label{wn_n_eq}
w_n^{(r)} - n = \left|\left\{k \in \mathbb{N} : m_k^{(r)} < w_n^{(r)}\right\}\right|.
\end{equation}

\noindent Note that we have $s_n^{(r)} = 1$ if and only if $w_n^{(r)} - n = 4k + 2$ for some $k \in \mathbb{N}$. Hence, from equation (\ref{wn_n_eq}), we get the formula
$$s_n^{(r)} = \left\{\begin{array}{ll}
1 & \text{ if } m_{4k+1}^{(r)} < w_n^{(r)} < m_{4k+2}^{(r)} \text{ for some } k \in \mathbb{N},\\
0 & \text{ otherwise}.
\end{array} \right.$$

\noindent Moreover, from Theorem \ref{unr_mnr_eq} and Theorem \ref{unr_differ_eq} we have that
$$m_{4k+2}^{(r)} - m_{4k+1}^{(r)} = 2^r - 1,$$

\noindent hence the sequence $(s_n^{(r)})_{n \in \mathbb{N}}$ consists of groups of consecutive $1$'s of the exact length $2^r - 2$. Moreover, if $n$ is a number such that $s_n^{(r)} = 1$ is the first one in the group of consecutive $1$'s, then $w_n^{(r)} = m_{4k+1}^{(r)} + 1$ for some $k \in \mathbb{N}$. From equation (\ref{wn_n_eq}), the last statement is equivalent to the fact that $n = m_{4k+1}^{(r)} - 4k - 1 = m_{4k}^{(r)} - 4k$. From equations (\ref{mnr_recur_eq}), it is clear that if $k = \sum_{i=0}^s2^{k_i}$ for some $0 \leq k_0 < \hdots < k_s$, then
$$m_k^{(r)} = \sum_{i=0}^s2^{rk_i}$$

\noindent and conversely. Therefore we have
$$n = m_{4k}^{(r)} - 4k = 2^{2r}\sum_{i=0}^s2^{rk_i} - 4\sum_{i=0}^s2^{k_i} = \sum_{i=0}^s(2^{r(k_i + 2)} - 2^{k_i + 2}),$$

\noindent and that completes the proof.
\end{proof}

\noindent We define a new sequence $(\tilde{s}_n^{(r)})_{n \in \mathbb{N}}$ as a sequence satisfying
$$\tilde{s}_n^{(r)} = \left\{\begin{array}{ll}
1 & \text{ if } n = \sum_{i=0}^s(2^{rn_i} - 2^{n_i}) \text{ for some } 1 < n_0 < \hdots < n_s,\\
0 & \text{ otherwise}.
\end{array} \right.$$

\noindent From Lemma \ref{lemma_sn_1}, we easily obtain the following result.
\begin{lem} \label{lemma_tildes}
We have the equality
$$\tilde{s}_n^{(r)} = \left\{\begin{array}{ll}
s_n^{(r)} & \text{ if } 2^r - 2 {\mid} n,\\
0 & \text{ otherwise}.
\end{array} \right.$$
\end{lem}

\noindent It turns out that properties of the sequence $(\tilde{s}_n^{(r)})_{n \in \mathbb{N}}$ are similar to properties of the sequence $(z_n)_{n \in \mathbb{N}}$ introduced in \cite{PTM}. The sequence satisfies the equality
$$z_n = \left(\frac{d_{4n} + 1}{4} - n\right) \!\! \pmod 2,$$ 

\noindent where $(d_n)_{n \in \mathbb{N}}$ is the characteristic sequence of $0$'s of the formal inverse of the Thue--Morse sequence.

\vspace{0.3 cm}
\noindent Before we give the proof of the main result, we introduce some additional lemmas concerning the sequence $(\tilde{s}_n^{(r)})_{n \in \mathbb{N}}$. These lemmas are analogous to the respective lemmas in \cite{PTM} concerning the sequence $(z_n)_{n \in \mathbb{N}}$ and they can be proved using exactly the same methods.

\begin{lem}
Let $k \in \mathbb{N}$, $k \geq 2$. Then for each $n \in \left[2^{r(k+1) - 1}, \frac{4}{3}\cdot 2^{r(k+1)-1}\right]$ we have $\tilde{s}_n^{(r)} = 0.$
\end{lem}

\begin{proof}
Analogous to \hspace{-0.2pt}{\cite[Lemma 4.2]{PTM}}.
\end{proof}

\begin{lem}
For all $k \in \mathbb{N}$, $k \geq 2$, there exists $n \in \mathbb{N}$ such that $k {\mid} n$ and $\tilde{s}_n^{(r)} = 1$.
\end{lem}

\begin{proof}
Analogous to \hspace{-0.2pt}{\cite[Lemma 4.3]{PTM}}.
\end{proof}

\begin{lem} \label{sn_2aut}
The sequence $(\tilde{s}_n^{(r)})_{n \in \mathbb{N}}$ is not $2$-automatic.
\end{lem}

\begin{proof}
Analogous to \hspace{-0.2pt}{\cite[Lemma 4.4]{PTM}}.
\end{proof}

\begin{lem} \label{sn_kaut}
Let $k > 1$ be a positive integer such that $k$ is not of the form $2^m$, $m \in \mathbb{N}$. Then the sequence $(\tilde{s}_n^{(r)})_{n \in \mathbb{N}}$ is not $k$-automatic.
\end{lem}

\begin{proof}
Analogous to \hspace{-0.2pt}{\cite[Lemma 4.5]{PTM}}.
\end{proof}

\vspace{0.2 cm}
\noindent We are now ready to prove the main result.
\begin{thm} The sequence $(v_n^{(r)})_{n \in \mathbb{N}}$ is not $k$-regular for any $k \geq 2$.
\end{thm}
\begin{proof} Suppose that the sequence $(v_n^{(r)})_{n \in \mathbb{N}}$ is $k$-regular for some $k \geq 2$. Then the sequence $(w_n^{(r)} - n)_{n \in \mathbb{N}}$ is also $k$-regular, hence the sequence $(s_n^{(r)})_{n \in \mathbb{N}}$ is $k$-automatic. From Lemma \ref{lemma_tildes}, we have that the sequence $(\tilde{s}_n^{(r)})_{n \in \mathbb{N}}$ is also $k$-automatic, which is a contradiction with Lemma \ref{sn_2aut} and Lemma \ref{sn_kaut}.
\end{proof}

\section{Open problems}

\noindent In this short section, we give some open problems connected with the Baum--Sweet sequence and the formal inverses of its modifications.

\begin{prob} Let $(h_n)_{n \in \mathbb{N}}$ be the characteristic sequence of $0$'s in the Baum--Sweet sequence, i.e.\ an increasing sequence satisfying the equality
$$\{h_n : n \in \mathbb{N}\} = \{m \in \mathbb{N} : b_m = 0\}.$$

\noindent Is the sequence $(h_n)_{n \in \mathbb{N}}$ $k$-regular for some $k$?
\end{prob}

\noindent In section \ref{sec_basic}, we proved that the sequence $(l_n)_{n \in \mathbb{N}}$, which is characteristic sequence of $1$'s in the Baum--Sweet sequence, is not $k$-regular for any $k$, which was the consequence of the fact that the frequency of $1$'s in the sequence $(l_n \!\!\! \mod 2)$ is not a rational number. Unfortunately, it turns out that if the frequency of $0$'s in sequence $(h_n \!\!\! \mod 2)$ exists, then it is rational.

\vspace{0.3 cm}
\noindent To prove this, we define the sequence of words $(H_n)_{n \in \mathbb{N}}$ in the following way:
$$H_0 = 0, \qquad H_1 = 10, \qquad H_{n+2} = H_n(01)^{2^n}H_{n+1}, \quad n \geq 2.$$

\noindent Denote by ${\bf h}$ the infinite word such that the $n$-th letter of ${\bf h}$ is equal to $h_n \!\! \mod 2$. Using the same method as in the proof of Theorem \ref{ln_mod2_eq}, we get that
$${\bf h} = H_0H_1H_2\dots$$ 

\noindent Using an easy induction, we get the equalities
$$|H_0H_1\dots H_n|_1 = 2^{n+1} - f_{n+3}, \qquad |H_0H_1\dots H_n| = 2^{n+2} - f_{n+4},$$

\noindent where $f_n$ denotes the $n$-th Fibonacci number. Hence we have
$$\lim_{n\to\infty} \frac{|H_0H_1\dots H_n|_1}{|H_0H_1\dots H_n|} = \lim_{n\to\infty}\frac{2^{n+1}-f_{n+3}}{2^{n+2}-f_{n+4}} = \frac{1}{2},$$

\noindent and therefore if the frequency of $1$'s in ${\bf h}$ exists, it is equal to $\frac{1}{2}$.

\begin{rem} 
It is not difficult to prove that the sequence $(h_n \!\!\!\mod 2)_{n \in \mathbb{N}}$ is morphic. Let $\Sigma = \{a,b,c,d,e,f\}.$ Let $\nu : \Sigma^{\star} \rightarrow \Sigma^{\star}$ be a morphism such that 
\begin{alignat*}{2}
\nu(a) &= bc, & \nu(d) &= de, \\
\nu(b) &= ad, & \nu(e) &= de, \\
\nu(c) &= ebc, & \hspace{0.7 cm} \nu(f) &= fbc.
\end{alignat*}

\noindent We define a coding $\tau : \Sigma \rightarrow \{0,1\}$ in the following way:
$$\tau(a) = \tau(c) = \tau(d) = \tau(f) = 0, \qquad \tau(b) = \tau(e) = 1.$$

\noindent Then it can be shown that we have the equality
$${\bf h} = \tau(\nu^{\omega}(f)).$$
\end{rem}

\vspace{0.3 cm}
\noindent The next problem is the generalization of the previous one.

\begin{prob} Let $r \geq 2$. Consider the sequence $(h_n^{(r)})_{n \in \mathbb{N}}$, which is the characteristic sequence of $0$'s in the sequence $(b_n^{(r)})_{n \in \mathbb{N}}$. Is the sequence $(h_n^{(r)})_{n \in \mathbb{N}}$ $k$-regular for some $k$?
\end{prob}

\noindent Let $p$ be a prime number. We denote by $s_p(n)$ the sum of digits of $n \in \mathbb{N}$ in base $p$. We define the sequence $(t_{p,n})_{n \in \mathbb{N}}$ as
$$t_{p,n} = s_p(n) \!\!\!\! \mod p \, \in \mathbb{F}_p.$$

\noindent The sequence $(t_{p,n})_{n \in \mathbb{N}}$ is then $p$-automatic. These sequences generalize the Thue--Morse sequence (we recover the original sequence in the case $p = 2$). 

\vspace{0.3 cm}
\noindent In section \ref{sec_inv2}, we proved that the formal inverse of the Thue--Morse sequence is closely related to the sequence $(q_n)_{n \in \mathbb{N}}$, which is the formal inverse of one of the modifications of the Baum--Sweet sequence. It would be interesting to find a similar connection involving the sequences $(t_{p,n})_{n \in \mathbb{N}}$.

\vspace{0.3 cm}
\noindent In order to find such a connection, we need to introduce a similar generalization of the Baum--Sweet sequence. Below we show one of the possible generalizations. It is unclear whether we can find such a connection in this particular case.

\vspace{0.3 cm}
\noindent We define the family of sequences $\{(b_{p,n})_{n \in \mathbb{N}} : p \text{ is prime}\}$ in the following way: for a fixed prime $p$, we have $b_{p,n} = 1$ if the length of every block of $0$'s in the base-$p$ expansion of $n$ is divisible by $p$ and $b_{p,n} = 0$ otherwise (with $b_{p,0} = 1$). The sequence $(b_{p,n})_{n \in \mathbb{N}}$ is $p$-automatic and for $p = 2$ we recover the original Baum--Sweet sequence. 

\vspace{0.3 cm}
\noindent We then introduce the sequence $(b'_{p,n})_{n \in \mathbb{N}}$ given by
$$b'_{p,n} = \left\{\begin{array}{ll}
0 & \text{ if } n = 0,\\
b_{p,n-1} & \text{ if } n \geq 1.
\end{array} \right.$$

\noindent The sequence $(b'_{p,n})_{n \in \mathbb{N}}$ has a formal inverse and therefore we can compare it to the formal inverse of the sequence $(t_{p,n})_{n \in \mathbb{N}}$.

\section*{Acknowledgements}
I would like to express my gratitude to my advisor, dr Jakub Byszewski, for numerous comments and suggestions.

\end{document}